\documentclass[12pt]{article}
\usepackage{amssymb,amsmath,amsthm}
\newtheorem{theorem}{Theorem}[section]
\newtheorem{lemma}[theorem]{Lemma}
\newtheorem{corollary}[theorem]{Corollary}
\newtheorem{proposition}[theorem]{Proposition}

\newtheorem{definition}[theorem]{Definition}

\def\ds{\displaystyle}
\def\dom{\mathop{\rm dom}}
\newcommand{\astorig}{(\ast)}
\newcommand{\astone}{(\ast_1)}
\newcommand{\astwo}{(\ast_2)}
\newcommand{\astonehat}{\left(\widehat{\ast_1}\right)}
\newcommand{\R}{\mathbb{R}}
\newcommand{\N}{\mathbb{N}}
\newcommand{\eps}{{\varepsilon}}
\newcommand{\be}{\begin{equation}}
\newcommand{\ee}{\end{equation}}
\renewcommand{\emptyset}{\varnothing}

\begin{document}


\title{On Long Orbit Empty Value (LOEV) principle}

\author{M. Ivanov, D. Kamburova, N. Zlateva}
\date{}

\maketitle

\begin{abstract}
We consider an useful in Variational Analysis tool -- Long Orbit or Empty Value (LOEV) principle -- in different settings, starting from more abstract to more defined.

We prove, using LOEV principle, a number of basic results in Variational Analysis, including some novel. We characterize $\Sigma_g$-semicompleteness for a generalized metric function $g$ which is neither symmetric nor satisfies the triangle inequality, in terms of validity of Ekeland Theorem for this $g$. We present an interesting application to perturbability to minimum in a $G_\delta$ subset of a complete metric space.
\end{abstract}

\section{Introduction}\label{sec:intro}

In \cite{Ivanov-Zlateva} it is postulated that the set-valued map $S: S \rightrightarrows S$, where $(X,d)$ is a complete metric space, satisfies the condition $(\ast)$, if $x \notin S(x),$ for all $x \in X$, and whenever $y \in S(x)$ and $x_n \to x$, as $n\to\infty$, there are infinitely many $x_n$'s such that
$y \in S (x_n)$.

It is shown that if $S$ satisfies $(\ast)$, then there is an $S$-orbit of infinite length, or there is  $x\in X$ such that $S(x)=\emptyset$, so, \emph{Long Orbit or Empty Value} (LOEV).

Further, the article \cite{Ivanov-Zlateva} shows that many set-valued maps naturally arising in Variational Analysis, satisfy $(\ast)$ and, therefore, many of the basic theorems of Variational Analysis can be proved using LOEV principle.

At present Ekeland Theorem is considered the basic one and all others are derived from it. Note that essentially Ekeland Theorem is an optimization result, while LOEV principle is inherently dynamical and can be applied where Ekeland Theorem is hard to apply. The origin of LOEV principle is in discretization of differential inclusions, see~\cite{iz-cled}.

Our aim here, however is somewhat different. To put it in context, LOEV principle can be considered as a non-idempotent counterpart to the famous Br\'ezis-Browder Principle, \cite{br}, see \cite{pbr} for a recent account of the huge influence of Br\'ezis-Browder Principle.

However, the non-idempotent version that can be stated at the same level of generality as Br\'ezis-Browder Principle, which is essentially a set-theoretic result, is rather cumbersome, so we skip it. Instead, in Section~\ref{sec:topology} we  present a general version of LOEV principle in a topological framework. It might not be the most general possible, but it extends the framework of \cite{Ivanov-Zlateva} considerably, including, for example, metric spaces that need not be complete.

The advantage of not requiring $S$ to be idempotent, as to generate partial ordering, is demonstrated in the most original Section~\ref{sec:semicomplete}, where we are able to derive, using LOEV principle, premetric versions of some of the most fundamental results of the Variational Analysis like the Theorems of Ekeland, Caristi and Takahashi. We also show that the validity of each of these characterizes the $\Sigma_g$-semicompleteness of the space. This section is inspired by \cite{Suzuki-2018}, but our approach through LOEV principle is original. Note also that no variant of Ekeland Theorem is formulated in \cite{Suzuki-2018}, so Theorem~\ref{th:Ekeland} is original.

Continuing with respecting the LOEV principle from different sides, from more general to more concrete settings, in Section~\ref{sec:cmp} we show how LOEV principle can be specified to work in complete metric space and we prove,  using this specification,  that is Theorem~\ref{thm:loev-mon}, the full Ekeland Theorem, as well as -- in order to demonstrate the technique -- two more complex theorems: those of Oettli and Th\'era, and Fabi\'an and Preiss.

Finally, in Section~\ref{sec:gd} we present a nontrivial application of the premetric version of Ekeland Theorem to perturbation within a $G_\delta$ subset of a complete metric space.

\section{LOEV principle in a topological framework}\label{sec:topology}

First, we set the formal framework for the present study.

We work in a first countable topological space, usually denoted by $(X,\tau)$. We consider a set-valued map $S:X\rightrightarrows X$ with domain 
$$
    \dom S := \{x\in X:\ S(x) \neq \emptyset\}.
$$
Recall that a (finite or infinite) succession of points in $X$ satisfying $x_{i+1}\in S(x_i)$ for $i=0,1,2,\dots $ is called an orbit of $S$, or $S$-orbit, starting at $x_0$. We say that an $S$-orbit ends at $ x\in X$ if the orbit is finite and $x_n= x$, or the orbit is infinite and the sequence $\{x_i\}_{i=0}^\infty$ converges to $ x$. If an $S$-orbit is infinite and the sequence $\{x_i\}_{i=0}^\infty$ diverges, we say that it is a divergent $S$-orbit.

For convenience, we impose \emph{non-stationary} condition:
\begin{equation}
    \label{eq:non-st}
    x\not\in S(x),\quad\forall x\in X.
\end{equation}
The meaning of this condition is that the stationary sequence $x,x,x,\ldots$ is never an $S$-orbit. Obviously, the latter can be easily ensured when necessary, by considering instead of $S$ the map $x\rightrightarrows S(x)\setminus \{x\}$.

Recall, see \cite{Ivanov-Zlateva}, that $S$ satisfies property $\astorig$ if \eqref{eq:non-st} is fulfilled, and for each $y\in S(x)$ and each sequence $x_i\to x$, there is a subsequence $\{x_{i_k}\}_{k=1}^\infty$ such that $y \in S(x_{i_k})$ for all $k\in\mathbb{N}$.

The property $\astorig$ is very practical, as it can be seen from the examples in \cite{Ivanov-Zlateva}, but it can be easily relaxed in several directions, most obvious of which is to consider only sequences starting from a certain fixed point.

\begin{definition}
    \label{def:ast-1}
    The set-valued map $S:X\rightrightarrows X$, where $(X,\tau)$ is a first countable topological space, satisfies the property $\astone$ at $x_0\in X$, if  \eqref{eq:non-st} is fulfilled, and for each infinite $S$-orbit $\{x_i\}_{i=0}^\infty$ starting at $x_0$ and ending at $x$ (that is, $x_i\to x$) and each $y\in S(x)$, there is a subsequence $\{x_{i_k}\}_{k=1}^\infty$ such that
    $$
        y \in S(x_{i_k}),\quad\forall k\in\mathbb{N}.
    $$
\end{definition}

Next, we define what we require from the spaces we work on.

\begin{definition}
    \label{def:h-space}
    We will call the triple $(X,\tau,h)$ an $h$-\emph{space}, if $(X,\tau)$ is a first countable Hausdorff space and the generalized distance function $h:X\times X\to \R^+$ is such that

    (i) $h(x,y)=0 \Leftrightarrow  x=y$;

   (ii) if $h(x,x_n)\to 0$, then $x_n\to x$;

   (iii) if $x_n\to x$, then $h(x_{n+1},x_{n})\to 0$.

\noindent (Of course, all convergences are as $n\to\infty$.)
\end{definition}

What follows is our first main result, which is a far-reaching generalization of LOEV principle.
\begin{theorem}
    \label{thm:loev-topo}
Let $(X,\tau,h)$ be an h-space. Let $S:X\rightrightarrows X$ be a set-valued map satisfying $\astone$ at $x_0$. Then at least one of (a) and (b) below is true:

(a) There is a divergent $S$-orbit starting at $x_0$;

(b) There is an $S$-orbit starting at $x_0$ and ending at $  x\in X$ such that $S( {x})=\varnothing$.
\end{theorem}
Before giving the proof of Theorem~\ref{thm:loev-topo} we prove the following results.
\begin{lemma}
    \label{lem:loev-induction}
    Let $X$ be a set. Let $S:X\rightrightarrows X$ be a set-valued map satisfying \eqref{eq:non-st}. Let $h:X\times X\to\mathbb{R}^+$ satisfy    \emph{Definition~\ref{def:h-space}(i)}. Let $x_0\in X$.

    Then at least one of the following two below holds:

        $\bullet$ There is a finite $S$-orbit $x_0,x_1,\ldots,x_n$ such that $S(x_n) = \emptyset$;
        
        $\bullet$ There is an infinite $S$-orbit $\{x_i\}_{i\ge0}$ such that
\begin{equation}\label{eq:x_i_i+1_jump}
            h(x_{i+1},x_i) > \min\{1,\sup h\left(S(x_i),x_i\right)\}/2.
        \end{equation}
\end{lemma}
\begin{proof}
    We can construct the desired finite or infinite $S$-orbit by the following procedure: if $x_0,x_1,\dots,x_i$ are already chosen, then

$\bullet$ If  $S(x_i)=\varnothing$  we are done.

$\bullet$  Otherwise, because of \eqref{eq:non-st} and (i),
$$
    \sup h(S(x_i),x_i) > 0.
$$
Choose a $x_{i+1}\in S(x_i)$ such that \eqref{eq:x_i_i+1_jump} holds.
\end{proof}
\begin{lemma}
    \label{lem:to0-empty}
    Let $(X,\tau,h)$ be an $h$-space. Let $S:X\rightrightarrows X$ be a set-valued map satisfying $\astone$ at $x_0$. Let $\{x_i\}_{i=0}^\infty$ be a convergent $S$-orbit: $x_{i+1}\in S(x_i)$ for all $i$, and $x_i\to\bar x$, as $i\to\infty$. If
    \begin{equation}
        \label{eq:sup-diam-0}
        \lim_{i\to\infty} \sup h\left(S(x_i),x_i\right) = 0,
    \end{equation}
    then $S(\bar x) = \emptyset$.
\end{lemma}
\begin{proof}
    Assume the contrary, that is, there is $y\in S(\bar x)$. From $\astone$ it follows that there is a subsequence $\{x_{i_k}\}_{k=1}^\infty$ such that $y\in S(x_{i_k})$ for all $k$. From \eqref{eq:sup-diam-0} it follows that $h(y,x_{i_k})\to0$, as $k\to\infty$, and (ii)  yields $x_{i_k}\to y$. But as a subsequence of a convergent to $\bar x$ sequence, $x_{i_k}\to \bar x$. Since $X$ is a Hausdorff space, $y= \bar x$.

    That is, $\bar x \in S(\bar x)$. Since $\astone$ includes \eqref{eq:non-st}, this is a contradiction.
\end{proof}

\begin{proof}\textbf{of Theorem~\ref{thm:loev-topo}.}
Assume that (\emph{a}) is not true. From Lemma~\ref{lem:loev-induction} it follows that either there is a finite $S$-orbit $x_0,x_1,\ldots,x_n$ such that $S(x_n) = \emptyset$ and we are done, or there is an infinite $S$-orbit that satisfies \eqref{eq:x_i_i+1_jump}, and converges to some $\bar x\in X$. Then $h(x_{i+1},x_i)\to 0$, because of (iii), so \eqref{eq:x_i_i+1_jump} implies \eqref{eq:sup-diam-0} and Lemma~\ref{lem:to0-empty} gives $S(\bar x) = \emptyset$.
\end{proof}

Next, we will analyze a bit the definition of an $h$-space.It is remarkable that generalized distance function $h$ actually defines the topology.
\begin{proposition}
    \label{prop:h-def-top}
    Let $(X,\tau,h)$ be an h-space. Then
    \begin{equation}
        \label{eq:ii-prim}
        x_n\to x \iff h(x,x_n)\to 0.
    \end{equation}
\end{proposition}
\begin{proof}
    Because of  Definition~\ref{def:h-space}(ii), we have to prove only that $x_n\to x$ implies $h(x,x_n)\to0$. Assume this were not the case and let $x_n$ be such a sequence that $x_n\to x$, but $\limsup_{n\to\infty} h(x,x_n)>0$. Consider the sequence
    $$
        x,x_1,x,x_2,\ldots,x,x_n,\ldots.
    $$
    It also converges to $x$, but Definition~\ref{def:h-space}(iii) fails. Contradiction.
\end{proof}

Because of Proposition~\ref{prop:h-def-top}, in an $h$-space $(X,\tau,h)$, for each positive integer $n$ and each $x\in X$, the level set
\begin{equation}
    \label{eq:v_n-h-def}
    V_n(x) := \{y\in X:\ h(x,y) < 1/n\}
\end{equation}
is a neighbourhood of $x$, that is, there is an open  set $U\in \tau$ such that $x\in U$ and $U\subset V_n(x)$, or, in other words, $x$ is in the interior of $V_n(x)$. In the terminology of \cite{Junnila}, the set-valued mapping
$$
    x\rightrightarrows V_n(x)
$$
is a \textit{neighbournet}, so $(X,\tau)$ has a countable family of neighbournets that define the topology in the sense that $U\in\tau$ if and only if for each $x\in U$ there is $n\in\mathbb{N}$ such that $V_n(x)\subset U$.

Since the function $h$ is not a part of the conclusion of Theorem~\ref{thm:loev-topo}, it is indeed possible to formulate the assumptions in purely topological terms. However, we started with $h$-spaces in order to keep this section consistent with the next sections, which are more analytical. Here is the straightforward characterization of when a topological space can be turned into an $h$-space.

\begin{proposition}
    \label{pro:char-h-space}
    Let $(X,\tau)$ be a first countable Hausdorff space. Then there exists a generalized distance function $h:X\times X\to \mathbb{R}$ such that $(X,\tau,h)$ 
    is an $h$-space if and only if there exists a countable nested family of neighbournets $\{V_n(x)\}_{n\in\mathbb{N}, x\in X}$, 
    which defines $\tau$, that is,
    \begin{equation}
        \label{eq:V-def-tau}
        \bigcap_{n=1}^\infty V_n(x) = \{x\},\quad\forall x\in X,
    \end{equation}
    and if $x_n\to x$ and $i_n\in\mathbb{N}$ are such that $x_n\notin V_{i_n}(x_{n+1})$, then
    \begin{equation}
        \label{eq:in-to-infty}
        \lim_{n\to\infty} i_n = \infty.
    \end{equation}
\end{proposition}
\begin{proof}
    Assume that $(X,\tau,h)$ is an $h$-space and define $V_n(x)$ via \eqref{eq:v_n-h-def}. It is clear that these are nested and we have already seen that they define $\tau$, but let's check \eqref{eq:V-def-tau} formally. If $y$ is in the intersection in the left hand side of \eqref{eq:V-def-tau}, then by definition $h(x,y) < 1/n$ for all $n$'s, so $h(x,y)=0$ and $x=y$ by Definition~\ref{def:h-space}(i).

       Assume now that \eqref{eq:in-to-infty} is false, so there is $x_n\to x$ and $k\in\mathbb{N}$ such that $x_n\notin V_k(x_{n+1})$ for infinitely many $n$'s, that is, $h(x_{n+1},x_n)\ge 1/k$ for infinitely many $n$'s, see \eqref{eq:v_n-h-def}. This contradicts Definition~\ref{def:h-space}(iii).

    For the opposite direction, let $\{V_n(x)\}_{n\in\mathbb{N}, x\in X}$ be a nested countable family of neighbournets in $(X,\tau)$ satisfying \eqref{eq:V-def-tau} and \eqref{eq:in-to-infty}. Set $V_0(x) := X$ and define $h:X\times X\to\mathbb{R}^+$ by $h(x,x):=0$, and
    $$
        h(x,y) := \frac{1}{n},\quad\forall y \in V_{n-1}(x)\setminus V_n(x),\ \forall n\in\mathbb{N}.
    $$
    Now, \eqref{eq:V-def-tau} implies  Definition~\ref{def:h-space}(i).

    If $h(x,x_n)\to 0$ then for each fixed $k\in\mathbb{N}$,  $h(x,x_n)<1/k$ for all but finitely many $x_n$'s, so, by definition,  all but finitely many $x_n$'s will be in $V_k(x)$, meaning that $x_n\to x$, and  Definition~\ref{def:h-space}(ii) is verified.

    Consider an arbitrary convergent sequence $x_n\to x$. Fix $\varepsilon > 0$. If $h(x_{n+1},x_n) = 0$
    for all but finitely many $n$'s, then $h(x_{n+1},x_n) = 0$ eventually and   Definition~\ref{def:h-space}(iii) is in force. If not, by collating the indices we may assume without loss of generality that $x_{n+1}\neq x_n$ for all $n\in\mathbb{N}$. Set $i_n :=1/h(x_{n+1},x_n)$. By definition  $x_n\notin V_{i_n}(x_{n+1})$ and \eqref{eq:in-to-infty} gives $h(x_{n+1},x_n)\to0$.
\end{proof}
There is much more to be considered with respect to topology. For example, we have not touched the dual topology. However,  we do not have space for that here, so we proceed to our next topic.

Often an intermediate property -- between $(\ast)$ and $\astone$ -- will be satisfied and it is worth considering.

\begin{definition}
    \label{def:ast-2}
    The set-valued map $S:X\rightrightarrows X$, where $(X,\tau)$ is a first countable topological space, satisfies the property $\astwo$ at $x_0\in X$, if \eqref{eq:non-st} is fulfilled for $S$,  and for each infinite $S$-orbit $\{x_i\}_{i=0}^\infty$ starting at $x_0$, each subsequence $\{x_{i_k}\}_{k=1}^\infty$ converging to $x$, and each $y\in S(x)$, there is a subsequence $\{x_{i_{k_j}}\}_{j=1}^\infty$ such that
    $$
        y \in S(x_{i_{k_j}}),\quad\forall j\in\mathbb{N}.
    $$
\end{definition}
Obviously, $\astorig\Rightarrow\astwo\Rightarrow\astone$. More precisely, $\astwo$ at $x_0$ implies $\astone$ at $x_0$, while $\astorig$ implies $\astwo$ for all $x$.

 Under this stronger assumption $\astwo$, we can specify the nature of potential divergent orbit in Theorem~\ref{thm:loev-topo}. Recall that a set is \emph{discrete} if it is closed and all its points are isolated. Clearly, a sequence in a first countable space is discrete if and only if it has no convergent subsequence. Naturally, a discrete orbit is such orbit which is discrete when considered as a set.

\begin{proposition}
    \label{pro:loev-topo}
    Let $(X,\tau,h)$ be an $h$-space. Let $S:X\rightrightarrows X$ be a  set-valued map satisfying $\astwo$ at $x_0$. Then at least one of the following is true:

    (a) There is an $S$-orbit starting at $x_0$, say $\{x_i\}_{i\ge0}$, such that its closure is not entirely contained in $\dom S$, i.e.,
    $$
        \overline{\{x_i\}_{i\ge0}}\not\subset\dom S.
    $$

    (b) There is a divergent $S$-orbit starting at $x_0$, say $\{x_i\}_{i=0}^\infty$, such that
    $$
        \limsup_{i\to\infty} h(x_{i+1},x_i) > 0.
    $$

    (c) There is an infinite discrete $S$-orbit starting at $x_0$.
\end{proposition}
\begin{proof}
    Let $x_0,x_1,\dots$ be the $S$-orbit given by Lemma~\ref{lem:loev-induction}. If it is finite,  then its end point is outside of $\dom S$ and (\emph{a}) holds. So, assume that it is an infinite sequence: $\{x_i\}_{i=0}^\infty$. If it has no convergent subsequences, we are done, because (\emph{c}) holds.

    So, let $x_{i_k} \to x$, as $k \to \infty$. If $S(x)=\emptyset$ then (\emph{a}) holds, so let $y\in S(x)$. From $\astwo$ there is further subsequence $\{x_{i_{k_j}}\}_{j=1}^\infty$ such that $y\in S(x_{i_{k_j}})$. If (\emph{b}) is not true, then
    $$
        \min\{1, h(y,x_{i_{k_j}})\} \le  2 h(x_{i_{k_j}+1}, x_{i_{k_j}}) \to 0,
    $$
    as $j\to\infty$. By  Definition~\ref{def:h-space}(ii) this implies $x_{i_{k_j}} \to y$, as $j\to\infty$, and, therefore, because $X$ is Hausdorff, $x=y$, contradiction.
\end{proof}

As it is evident from \cite{Ivanov-Zlateva}, and as we will see again later on here, many useful maps are \emph{idempotent}, that is,
\begin{equation}
    \label{eq:mono}
    S^2(x) = S(S(x)) \subset S(x),\quad \forall x\in X.
\end{equation}
Another way of defining the idempotency of $S$ is to say that
$$
    y \preceq x \iff y \in \{x\}\cup S(x)
$$
defines a partial ordering on $X$. It will be noted later that most of the applications found in the literature rely on introducing some partial ordering. We, on the other hand, need not impose idempotency on $S$. 

If $S$ is idempotent, then, of course, each subsequence of an infinite $S$-orbit is again an $S$-orbit, and, therefore for an idempotent map $\astone$ is equivalent to $\astwo$, so the latter is not as artificial as it may seem at first glance.

Also, for an idempotent $S$, Proposition~\ref{pro:loev-topo}(\emph{a}) can be reformulated as: \emph{There is an $S$-orbit starting at $x_0$ and ending outside of $\dom S$}.

\section{LOEV principle in  $\Sigma_g$-semicomplete  premetric space}\label{sec:semicomplete}

Let $(X,\tau)$ be a  topological space. A generalized distance function $g:X \times X \rightarrow \R^+$ with the following properties:
\begin{itemize}
     \item[{\rm (P1)}] $g(x,x)=0$ for any $x\in X$ and $g(x,y) > 0$ whenever $x \neq y$;
     \item[{\rm (P2)}] $g(x, \cdot)$ is continuous for every fixed $x\in X$
\end{itemize}
will be called a \emph{premetric function on} $X$. Obviously, (P1) is the same as    Definition~\ref{def:h-space}(i), but we repeat it here for ease of reference. If there exists a premetric function $g$ on a topological space $X$, then  necessarily $X$  is a Hausdorff space. A first countable topological space $(X,\tau)$ with premetric function $g$ on it will be called \emph{premetric space} and will be denoted by $(X,\tau,g)$.

Premetric functions exist on very large classes of spaces. For example,
let $X$ be a completely regular and first countable space. Fix a point $x \in X$ and let  $\{U_k\}_{k\in\mathbb{N}}$ be a  nested local base at $x$. For a $k\in \N$ let $p_k: X \rightarrow [0, 1]$ be a continuous function such that $p_k(x)=0$ and $p_k|_{X \setminus U_k} \equiv 1$. Define
\begin{equation}
    \label{eq:g-def}
    g(x,y):=\sum_{k=1}^{\infty}\frac{p_k(y)}{2^k}.
\end{equation}
The function $g: X \times X \rightarrow [0,1]$ is a premetric function that, moreover, satisfies  Definition~\ref{def:h-space}(ii).

$\Sigma$-Cauchy sequences and $\Sigma$-semicompleteness are considered in Suzuki~\cite{Suzuki-2018}. We extend these notions  to a premetric space $(X,\tau,g)$ in the following way.

For a sequence $\{x_n\}_{n\ge 0} $ in a premetric space $(X,\tau,g)$,  $\ds  \sum_{i=0}^{\infty} g(x_{i+1},x_i)$ can be considered as the $g$-length of the sequence. If a sequence $\{x_n\}_{n\ge 0} $ has a finite $g$-length, it will be called  $\Sigma_g$-Cauchy sequence. The notion of $\Sigma_g$-Cauchy sequence also appears as absolutely convergent sequence in \cite{MacNeille}.

\begin{definition} \label{completeness}
A first countable premetric space $(X,\tau,g)$ is called  $\Sigma_g$-semicomplete if every $\Sigma_g$-Cauchy sequence in $X$ has a convergent subsequence.
\end{definition}

Next we will prove a variant of LOEV principle in a  $\Sigma_g$-semicomplete space.

\begin{theorem}\label{LOEV}
Let $(X,\tau,g)$ be a $\Sigma_g$-semicomplete space, let $S:X \rightrightarrows X$ satisfy $\astwo$ at $x_0 \in X$. Then at least one of (a) and (b) below is true:

(a) There is an $S$-orbit starting at $x_0$ with an infinite $g$-length;

(b) There is an $S$-orbit starting at $x_0$ with a finite $g$-length and such that its closure is not contained in the domain of $S$.
\end{theorem}

\begin{proof}
Let $x_0,x_1,\dots$ be the orbit given by Lemma~\ref{lem:loev-induction} for $S$ and $g$. If it is finite, or of infinite $g$-length, we are done. So, assume that it constitutes a $\Sigma_g$-Cauchy sequence $\{x_i\}_{i \geq 0}$. Obviously then $g(x_{i+1},x_i) \to 0$, as $i\to\infty$.  From \eqref{eq:x_i_i+1_jump} it follows that
\begin{equation}
    \label{eq:g-S-goes-to-0}
    \lim_{i\to\infty}\sup g(S(x_i),x_i) = 0.
\end{equation}
 From the  $\Sigma_g$-semicompleteness of the space $(X,\tau,g)$ it follows that  there is a convergent subsequence $\{x_{i_k}\}_{k\in\mathbb{N}}$. Let  $x_{i_k} \rightarrow  x\in X$, as $k \rightarrow \infty$. If $S(x) = \emptyset$, we are done, so assume there is $y \in S({x })$. By \eqref{eq:non-st} we have $y\neq x$. Since  $g$ has property (P1), we have that $g(y,x)>0$.

 By the property $\astwo$ of $S$ at $x_0$, there is a further subsequence $\{x_{i_{k_j}}\}_{j\in\mathbb{N}}$ such that  $y \in S(x_{i_{k_j}})$ for all $j\in \N$. By (P2) $g(y,\cdot)$ is continuous, so when $j \rightarrow \infty$, $g(y, x_{i_{k_j}}) \rightarrow g(y, x)>0$, . But then $\sup g(S(x_{i_{k_j}}),x_{i_{k_j}}) \ge g(y, x_{i_{k_j}}) > g(y, x)/2 > 0$ for all large enough $j$'s, which contradicts \eqref{eq:g-S-goes-to-0}.
\end{proof}

The next result shows that if LOEV holds  in a premetric space $(X,\tau,g)$, then necessarily the space is  $\Sigma_g$-semicomplete.

\begin{theorem}\label{LOEV inverse} 
Let  $(X,\tau,g)$ be a first countable premetric space. If for  any $x_0\in X$ and any $S:X \rightrightarrows X$ that satisfies $\astwo$ for   $x_0 $ at least one of (a) and (b) below is true:

(a) There is an $S$-orbit starting at $x_0$ with infinite $g$-length;

(b) There is ${x} \in X$ such that $S({x})=\varnothing$,

\noindent
then the space $(X,\tau,g)$ is  $\Sigma_g$-semicomplete.
\end{theorem}
In the sequel we will essentially use the following result which holds in a space not $\Sigma_g$-semicomplete.
\begin{lemma}
    \label{rem:1}
    Let  $(X,\tau,g)$ be a first countable premetric space which is not $\Sigma_g$-semicomplete. Then there exists a countable discrete set
    $$
        M = \{x_n:\ n\in\mathbb{N}\},
    $$
    such that the sequence $\{x_n\}_{n=1}^\infty$ has a finite $g$-length.

    Moreover, there is a strictly positive lower semicontinuous function $f:X\to\mathbb{R}^+\setminus\{0\}$ such that $\dom f \equiv M$, $\inf f = 0$, and
    \begin{equation}
        \label{eq:f-g-M-rel}
        f(x_n) - f(x_{n+1}) = g(x_{n+1},x_n), \quad\forall n\in\mathbb{N}.
    \end{equation}
\end{lemma}
\begin{proof}
    By definition, there is a $\Sigma_g$-Cauchy sequence $\{ z_k\}_{k\ge 0}$ with no convergent subsequences. Let for $n\in\mathbb{N}$
    $$
        \sigma(n) := \max \{k:\ z_k = z_n\}.
    $$
    Note that the maximum is well defined, because if there were infinitely many equal $z_k$'s they would have made a convergent (even constant) subsequence. So,
    $$
        n \le \sigma(n) < \infty.
    $$
    Clearly,
    \begin{equation}
        \label{eq:sigma-alt-def}
        z_{\sigma(n)} = z_n\text{ and }z_k\neq z_n,\ \forall k > \sigma(n).
    \end{equation}
    We construct by induction a subsequence $\{ z_{k_n}\}_{n\in\mathbb{N}}$ in the following way:
     $$
        k_1 := 1,\quad  k_{n+1} := \sigma(k_n) + 1.
    $$
    Clearly, $\{k_n\}_{n\in\N}$ is strictly increasing, because $\sigma(n)\ge n$, so $\{z_{k_n}\}_{n\in\N}$ is a subsequence of  $\{ z_k\}_{k\ge 0}$, thus it has no convergent subsequences. Set
    $$
        x_n := z_{k_n}, \quad\forall n\in\mathbb{N}.
    $$
    Then $\{x_n\}_{n\in\mathbb{N}}$ has no convergent subsequences and also
    $$
        x_i \neq x_j,\quad\forall i\neq j \in \mathbb{N}.
    $$
    Indeed, let $i < j$. Since $\{k_n\}_{n\in\N}$ is  increasing, $k_j \ge k_{i+1} =  \sigma (k_i) + 1 > \sigma (k_i)$, so $x_j = z_{k_j} \neq z_{k_i}=x_i$, see \eqref{eq:sigma-alt-def}. This means that the range $M$ of the sequence $\{x_n\}_{n\in\mathbb{N}}$ is exactly as defined in the Lemma, or, in other words, the map $n\to x_n$ is injective. The set $M$ is discrete, because  $\{x_n\}_{n\in\mathbb{N}}$ has no convergent subsequences.

    To estimate the $g$-length of $\{x_n\}_{n\in\mathbb{N}}$, note that
    $$
        x_n = z_{k_n} = z_{\sigma(k_n)},\text{ and }x_{n+1} = z_{\sigma(k_n)+1},
    $$
    so
    \begin{eqnarray*}
        \sum_{n=1}^\infty g(x_{n+1},x_n) &=& \sum_{n=1}^\infty g(z_{\sigma(k_n)+1},z_{\sigma(k_n)}) \\
        &\le& \sum_{k=0}^\infty g(z_{k+1},z_k) < \infty,
    \end{eqnarray*}
    therefore the $g$-length of $\{x_n\}_{n=1}^\infty$ is finite.

    Define the function
\be\label{eq:f}
f(x):= \left \{\begin{array}{rl} \sum_{i=n}^{\infty} g(x_{i+1},x_{i}), & x = x_{n};  \\
+\infty, & x\notin M.\end{array} \right.
\ee
It is obviously a   strictly positive function. Since the sums are finite, $\dom f\equiv M$. The equality \eqref{eq:f-g-M-rel} is satisfied by definition. Clearly, $\lim_{n\to\infty} f(x_n) = 0$, so $\inf f = 0$. Since $M$ consists of isolated points, $f$ is also lower semicontinuous.
\end{proof}

\begin{proof} \textbf{of Theorem~\ref{LOEV inverse}}. Suppose that $(X,\tau,g)$ is not  $\Sigma_g$-semicomplete. Let the $\Sigma_g$-Cauchy sequence $\{x_n\}_{n=1}^\infty$ and its range $M$ be given by Lemma~\ref{rem:1}. Define the set-valued map  $S$ in the following way:
\begin{equation}
    \label{eq:S-M-def}
    S(x):= \left \{\begin{array}{rl} M, & \text{if } x \notin M;  \\
x_{n+1} ,& \text{if } x = x_{n} \text{ for some } x_{n}  \in M. \end{array} \right.
\end{equation}
We claim that $S$ satisfies $(\ast)$, ergo $\astwo$ at $x_1$.

To this end, it is clear that \eqref{eq:non-st} holds at $S$. Let $y_n\to y$ be given and $z\in S(y)$. If $y\not\in M$, then because the latter is closed, all but finitely many $y_n$'s will be outside of $M$, so $S(y_n) = M\ni z$ for all but finitely many $y_n$'s. If now $y\in M$, say $y = x_k$, then $z=x_{k+1}$. Since $x_k$ is an isolated point of $M$, all but finitely many $y_n$'s will be outside of the set $M\setminus\{x_k\}$, so $z=x_{k+1}\in S(y_n)$ for all but finitely many $y_n$'s. Therefore, $S$ does satisfy $(\ast)$.

By definition, $S(x) \neq \varnothing$ for all $x \in X$, therefore (\emph{b}) does not hold. But  the unique $S$-orbit starting at $x_1$ is $\{x_n\}_{n=1}^\infty$, and it is $\Sigma_g$-Cauchy, so (\emph{a}) does not hold either. This contradiction completes the proof.
\end{proof}

The above result shows that the LOEV principle is, roughly speaking, equivalent to the
 $\Sigma_g$-semicompleteness of the first countable premetric space.

 Further we will prove variants of the Caristi Theorem, Takahashi Theorem and Ekeland Variational Principle in  $\Sigma_g$-semicomplete space and establish that each of them is equivalent to $\Sigma_g$-semicompleteness.

 We begin with a variant of the Caristi Theorem, see \cite{Caristi}.

\begin{theorem} \label{th:Caristi}
Let $(X,\tau,g)$ be a   $\Sigma_g$-semicomplete space. Let the set-valued map $T:X \rightrightarrows X$ satisfy $T(x) \neq \varnothing$ for all $x \in X$. Let the function $f: X \rightarrow \R \cup \{+\infty\}$ be proper, lower semicontinuous and bounded below. Assume that for any $x \in X$ there exists $y \in T(x)$ such that
\begin{equation}\label{CK_condition}
g(y,x) \leq f(x)-f(y). 
\end{equation}
Then $T$ has a fixed point, that is, there exists $\bar{x}\in X$ such that $\bar{x} \in T(\bar{x})$.
\end{theorem}

\begin{proof}
Assume that $x \notin T(x)$ for all $x \in X$. Consider the map
\[S(x):=\{y \in X: 2^{-1}g(y,x) < f(x)-f(y) \}.\]
We claim that $S$ satisfies $(\ast)$. Indeed, $S$ satisfies \eqref{eq:non-st} by definition. Let $x_n\to \bar x$ and $y\in S(\bar x)$, so
$$
   f(\bar x) - f(y) - g(y,\bar x)/2 > 0.
$$
Since $f$ is lower semicontinuous and $g$ is continuous on its second argument by (P2), the function $x\to f(x) - f(y) - g(y,x)/2 >0$ is lower semicontinuous at $\bar x$. So,  $f(x_n) - f(y) - g(y,x_n)/2 >0$ for all but finitely many $n$'s, that is, $y\in S(x_n)$ for all but finitely many $n$'s.

Now, since for each $x \in X$ there is some $y \in T(x)$, $y\neq x$, satisfying (\ref{CK_condition}) and by (P1), $g(y,x)>0$, we have that  $S(x) \neq \varnothing$ for all $x \in X$.

We apply  Theorem~\ref{LOEV} to get an $S$-orbit $\{x_i\}_{i\ge 0}$ with an
infinite $g$-length.

From the  the definition of $S$ it follows that  $2^{-1}g(x_{i+1},x_{i})<f(x_i)-f(x_{i+1})$ for all $i\ge 0$. Summing the inequalities we obtain that
\[ 2^{-1} \sum_{i=0}^{n-1} g(x_{i+1},x_{i})<f(x_0)-f(x_{n}), \quad \forall n\ge 1.\]
 Letting  $n$ to infinity, we get $f(x_n) \rightarrow -\infty$ which contradicts the boundedness below of  $f$. The proof is then complete.
\end{proof}

Now we will prove that if a variant of Caristi Theorem  holds  in a first countable premetric space $(X,\tau,g)$, then it is  necessarily $\Sigma_g$-semicomplete.

\begin{theorem} \label{Caristi inverse}
Let $(X,\tau,g)$ be a first countable premetric space. If every set-valued map $T: X \rightrightarrows X$ satisfying Caristi condition (\ref{CK_condition}) with some proper lower semicontinuous and bounded below function $f:X\to \R\cup\{+\infty\}$, has a fixed point, then $(X,\tau,g)$ is  $\Sigma_g$-semicomplete.
\end{theorem}
\begin{proof}
Suppose that $(X,\tau,g)$ is not  $\Sigma_g$-semicomplete, and let the positive lower semicontinuous $f$ with $\dom f = M\equiv \{x_n:\ n\in\mathbb{N}\}$ be given by Lemma~\ref{rem:1}.

Consider the map  $S$ from \eqref{eq:S-M-def}. It has no fixed point, but satisfies the Caristi condition \eqref{CK_condition} with  the function $f$. Indeed, if $x\notin M$ then $x_1\in S(x)$ and $f(x)-f(x_1)=\infty$. If $x=x_n$ then $x_{n+1}\in S(x_n)$ and we note \eqref{eq:f-g-M-rel}.
\end{proof}

In a $\Sigma_g$-semicomplete space we can prove also a variant of Takahashi Theorem, see~\cite{Takahashi}.

\begin{theorem}\label{th:Takahashi} Let $(X,\tau,g)$ be a $\Sigma_g$-semicomplete space. Let the function $f: X \rightarrow \R \cup \{+\infty\}$ be proper, lower semicontinuous and bounded below. Suppose that for each $x \in X$ with $f(x) > \inf_X f$ there exists $y \in X$, $y \neq x$, such that
\begin{eqnarray}
g(y,x) \leq f(x)-f(y). \label{eq:Takahashi}
\end{eqnarray}
Then, there exists $v \in X$ such that $f(v)=\inf_X f$.
\end{theorem}

\begin{proof}
Assume that the condition holds but $f(x) > \inf_X f$ for all $x \in X$. This means that for any $x\in X$  there exists $y \in X$, $y \neq x$, such that
\[
2^{-1}g(y,x) < g(y,x)\leq f(x)-f(y),
\]
where the strict inequality follows from (P1).
 Define for all $x \in X$,
\[S(x):=\{y \in X: 2^{-1}g(y,x) < f(x)-f(y)\} \neq \varnothing.\]
As above, $S$ satisfies $(\ast)$, so we can get an $S$-orbit $\{x_n\}_{n\ge0}$ of an infinite length.

From the definition of $S$ it holds that $2^{-1}g(x_{i+1},x_{i})<f(x_i)-f(x_{i+1})$ for all $i\ge 0$ and summing the inequalities and passing to infinity, we obtain that $f(x_n) \rightarrow -\infty$ which contradicts the boundedness below of $f$.
\end{proof}

As announced,  if the variant of the Takahashi Theorem  holds  in a first countable premetric space $(X,\tau,g)$, then necessarily the space is $\Sigma_g$-semicomplete.

\begin{theorem} \label{Takahashi invers}
Let $(X,g)$ be a first countable premetric space. If for every proper, lower semicontinuous and bounded below function $f: X \rightarrow \R \cup \{+\infty\}$ that satisfies  condition (\ref{eq:Takahashi}) there exists $v \in X$ such that $f(v)=\inf_X f$, then $(X,\tau,g)$ is   $\Sigma_g$-semicomplete.
\end{theorem}
\begin{proof} Assume that $(X,\tau,g)$ is   not $\Sigma_g$-semicomplete.
Let the strictly positive lower semicontinuous $f$ with $\dom f = M\equiv \{x_n:\ n\in\mathbb{N}\}$ be given by \eqref{eq:f} in Lemma~\ref{rem:1}. To check \eqref{eq:Takahashi}, note that if $x\notin M$ then $f(x) - f(x_1) = \infty > g(x_1,x)\in\mathbb{R}$; and if $x=x_n$, then \eqref{eq:f-g-M-rel} implies \eqref{eq:Takahashi}. Now, since $\inf f = 0$, there should be some $v$ with $f(v)=0$, but it is strictly positive.  Contradiction.
\end{proof}

Now we will prove  in a  $\Sigma_g$-semicomplete space  a variant of Ekeland Theorem~\cite{Ekeland}. For a more detailed version, see \cite{ka-z-conf}.

\begin{theorem} \label{th:Ekeland}
Let $(X,\tau,g)$ be a $\Sigma_g$-semicomplete space. Let the function $f: X \rightarrow \R \cup \{+\infty\}$ be proper, lower semicontinuous and bounded below. Then for every $\varepsilon>0$ there exists $v \in X$, such that
\begin{eqnarray}
f(v) \leq f(x)+\varepsilon g(x,v) \text{ for all } x\in X. \label{eq:Ekeland}
\end{eqnarray}
\end{theorem}

\begin{proof}
Assume that for some $\eps >0$ there is no point $v \in X$ such that \eqref{eq:Ekeland} holds. Then for every $x\in X$ there should be some $y\in X$ such that
$f(x)-\eps g(y,x)>f(y)$.

Consider the set-valued map $S$ defined as
\[S(x):=\{y \in X: f(y)<f(x)-\varepsilon g(y,x)\},\]
and observe that $S(x) \neq \varnothing$ for all $x \in X$. By (P1) property  of the function $g$, $x \not\in S(x)$ for all $x \in X$. The lower semicontinuity of $f$ and (P2) property of $g$ imply that $S$ satisfies $(\ast)$. So we can apply Theorem~\ref{LOEV} to get an $S$-orbit $\{x_n\}_{n\ge 0}$ with an infinite $g$-length.
But the definition of $S$ implies that $f(x_{i+1})-f(x_i)<-\varepsilon g(x_{i+1},x_{i})$ for all $i\ge 0$. Summing the first $n$ inequalities we get
\[
f(x_{n})-f(x_0)<-\varepsilon \sum_{i=0}^n g(x_{i+1},x_{i})
\]
and passing $n$ to  infinity we obtain that  $f(x_n) \rightarrow -\infty$. The latter contradicts the boundedness below of  $f$.
\end{proof}

Note that Ekeland Theorem is still valid if we impose weaker than lower semicontinuity assumption on $f$, namely, lower semicontinuous from above. A function $f: X \to \mathbb{R} \cup \{ + \infty \}$ is said to be \emph{lower semicontinuous from above} if for every sequence $\{x_n\}_{n=1}^\infty \subset \dom f$ such that $x_n \to x $ and $\{f(x_n)\}_{n=1}^\infty$ is decreasing, i.e. $f(x_n)>f(x_{n+1})$, it holds that $f(x) \leq f(x_n)$ for all $n \in \N$, see \cite{KS}. The notion of lower semicontinuity from above also appears as condition (H4) in  \cite[Subsection 3.11.1]{Tammer-book}. The proof of Ekeland Theorem for lower semicontinuous from above functions in premetric spaces is given in~\cite{ka-z-conf}.

Conversely,  if the variant of the Ekeland Theorem  holds  in a first countable premetric space $(X,\tau,g)$, then necessarily the space is $\Sigma_g$-semicomplete as we will show further.

\begin{theorem} \label{Ekeland invers}
Let $(X,\tau,g)$ be a first countable premetric space. If for every $\varepsilon>0$ and every proper lower semicontinuous and bounded below function $f: X \rightarrow \R \cup \{+\infty\}$ there exists $v \in X$ such that (\ref{eq:Ekeland}) holds, then $(X,\tau,g)$ is  $\Sigma_g$-semicomplete.
\end{theorem}

\begin{proof}
Suppose that $X$ is not $\Sigma_g$-semicomplete. Let the strictly positive lower semicontinuous $f$ with $\dom f = M\equiv \{x_n:\ n\in\mathbb{N}\}$ be given by Lemma~\ref{rem:1}.

From our assumption, \eqref{eq:Ekeland} holds for this $f$ with some $v\in X$ and $\varepsilon = 1/2$, that is
\[f (v) \leq f(x)+ g(x,v)/2,\quad\forall x\in X.\]
Note that necessarily $v \in \dom f \equiv M$. Therefore, there exists $n$ such that $v=x_{n}$. For  $x=x_{n+1}$, using the above and \eqref{eq:f-g-M-rel}, we get
$$
  f(x_{n+1}) + g(x_{n+1},x_{n}) = f(x_{n}) \leq f(x_{n+1})+ g(x_{n+1},x_{n})/2.
$$
Since $x_{n+1}\in \dom f$, and $g\ge0$, the latter implies $g(x_{n+1},x_{n})=0$, which due to (P1) contradicts $x_n\neq x_{n+1}$.\end{proof}

\section{LOEV principle in complete metric space}\label{sec:cmp}

    In this section we work in a complete metric space $(X,d)$ which is the original setting of LOEV principle. Obviously, if $\tau_d$ is the canonical topology generated by the metric, then $(X,\tau_d,d)$ is both a $h$-space and a $\Sigma_d$-semicomplete premetric space (it is, of course, complete and there is no difference between completeness and semicompleteness here) so, all results so far work in a complete metric space $(X,d)$.

    However, the presence of the triangle inequality and the full continuity of $d$ as a function from $X\times X$ to $\mathbb{R}^+$, add more structure, so we can prove a stronger result. We will also demonstrate how someone interested only in complete metric spaces could work somewhat easier than in the previous sections.

    To put what follows into context, recall that a set-valued map $S:X\rightrightarrows X$ is lower semicontinuous at $x\in X$ if for each $y\in S(x)$ there exists a sequence $x_n\to x$ and $y_n\in S(x_n)$ such that $y_n\to y$, see \cite[p.39]{auf}. Naturally, $S$ is lower semicontinuous if it is lower semicontinuous at each $x\in X$. We see that the property $(\ast)$ trivially implies lower semicontinuity. So, if we restrict the lower semicontinuity to the orbits starting at certain point, we obtain a property weaker than $\astone$, but also more symmetric than~$\astone$.
    \begin{definition}
        \label{def:hat-st-1}
        Let $(X,d)$ be a complete metric space. We say that the set-valued map $S:X\rightrightarrows X$ satisfies the condition $\astonehat$ at $x_0\in X$ if it satisfies \eqref{eq:non-st}, and for each infinite $S$-orbit $\{x_n\}_{n\ge0}$ starting at $x_0$ and ending at $x$, that is, $x_n\to x$, as $n\to\infty$; and each $y\in S(x)$, there exist a subsequence $\{x_{n_k}\}_{k\in\mathbb N}$ of $\{x_n\}_{n\ge0}$ and $y_k\in S(x_{n_k})$ such that $y_k\to y$, as $k\to\infty$.
    \end{definition}
    The following is a localization of the version of LOEV principle found in \cite{plr}. The reader is advised to check  \cite{plr} for a very specific application of LOEV principle, which apparently cannot be done by Ekeland Theorem.
    \begin{theorem}
        \label{thm:loev-cmp}
        Let $(X,d)$ be a complete metric space. Let the set-valued map $S:X\rightrightarrows X$ satisfy the condition $\astonehat$ at $x_0\in X$. Then at least one of the following two holds:
        \begin{enumerate}
            \item[(a)] There is an $S$-orbit of finite length starting at $x_0$ and ending outside of $\dom S$;
            \item[(b)] There is an $S$-orbit of infinite length starting at $x_0$.
        \end{enumerate}
    \end{theorem}
    \begin{proof}
        Let $\{x_n\}_{n\ge0}$ be the orbit given by Lemma~\ref{lem:loev-induction}. If its length is infinite, we are done.

        If, on the other hand, the $d$-length of $\{x_n\}_{n\ge0}$ is finite then, since $(X,d)$ is complete, it converges to, say, $x$. Since $d(x_{n+1},x_n)\to0$, we have by \eqref{eq:x_i_i+1_jump} that
        $$
            \sup d(S(x_n),x_n) \to 0,\text{ as }n\to\infty.
        $$
        If we assume that $y\in S(x)$ then $y\neq x$ by \eqref{eq:non-st}. By $\astonehat$ at $x_0\in X$ we have a subsequence $\{x_{n_k}\}_{k\in\mathbb N}$ and  $y_k\in S(x_{n_k})$ such that $y_k\to y$, as $k\to\infty$. Then eventually
        $$
            \sup d(S(x_{n_k}),x_{n_k}) \ge d(y_k, x_{n_k}) > d(y,x)/2 >0,
        $$
        contradiction.
    \end{proof}

    If we want to repeat the proof of Theorem~\ref{th:Ekeland} here, we would not be able to get the strict inequality in \eqref{eq:eke-2} below. The reason is the very definition of $S$, which would be:
    $$
        S(x) := \{y:\  f(y) < f(x) -\lambda d(y,x)\}.
    $$
    Obviously, we would like to define  $S(x)$ with $\le$, and this here is possible because the triangle inequality ensures that $S$ is idempotent.
    \begin{theorem}
        \label{thm:loev-mon}
        Let $(X,d)$ be a complete metric space. Let the set-valued map $S:X\rightrightarrows X$ be idempotent and with closed values. Let $x_0\in X$. Then at least one of the following two holds:
 
           (a) There is an $S$-orbit of finite length starting at $x_0$ and ending at such $x\in X$ that
            $$
                S(x) \subset \{x\};
            $$
            
           (b) There is an $S$-orbit of infinite length starting at $x_0$.
      \end{theorem}
    \begin{proof}
        Let $\{x_n\}_{n=0}^\infty$ be a convergent $S$-orbit starting at $x_0$ and ending at $\bar x\in X$, that is, $\bar x = \lim_{n\to\infty}x_n$. We claim that
        \begin{equation}
            \label{eq:s-in-cap}
            S(\bar x) \subset \bigcap_{n=0}^\infty S(x_n).
        \end{equation}
        Indeed, for each fixed $n\ge0$ the idempotency of $S$ gives $x_{n+1}\in S(x_n)$, $x_{n+2}\in S(x_{n+1})\subset S(S(x_n)) \subset S(x_n)$, and so on, $x_k\in S(x_n)$ for all $k\ge n$. Since $S(x_n)$ is closed, we have that $\bar x\in S(x_n)$, and by idempotency, $S(\bar x) \subset S(x_n)$. Since $n\ge0$ is arbitrary, we get \eqref{eq:s-in-cap}. Let
        $$
            S'(x) := S(x) \setminus \{x\},
        $$
        so $S'$ satisfies \eqref{eq:non-st}. If  $\{x_n\}_{n=0}^\infty$ is a convergent $S'$-orbit, say $x_n\to\bar x$, as $n\to\infty$, then $\{x_n\}_{n=0}^\infty$ is also an $S$-orbit and \eqref{eq:s-in-cap} gives that if $y\in S'(\bar x)$ then $y\in S(x_n)$ for all $n\ge0$. But $y\neq \bar x$, so $y\neq x_n$ for all but finitely many $n$'s, so $y\in S'(x_n)$ for all but finitely many $n$'s. We see, therefore, that $S'$ satisfies $\astone$ at $x_0$ and we can apply Theorem~\ref{thm:loev-cmp} to it.

        If (\emph{b}) of  Theorem~\ref{thm:loev-cmp} is satisfied for $S'$ then (\emph{b}) of the present statement is also satisfied, because each $S'$-orbit is also an $S$-orbit.

        If (\emph{a}) of  Theorem~\ref{thm:loev-cmp} is satisfied for $S'$ then we have a $S'$-orbit of a finite length ending at $\bar x$ such that $S'(\bar x) = \emptyset$, so $S(\bar x) \subset \{\bar x\}$.
    \end{proof}
    Ekeland Theorem is the only statement we will repeat from the previous section, because of its prominence.
    \begin{theorem}(Ekeland)
        \label{thm:eke-cmp}
        Let $(X,d)$ be a complete metric space.
        If $f: X \rightarrow \mathbb{R} \cup\{+\infty\}$ is proper, lower semicontinuous, and bounded below, and $x_0 \in \dom f$, then for each $\lambda>0$ there exists $x_\lambda$ such that
        \begin{equation}
            \label{eq:eke-1}
            \lambda d(x_\lambda, x_0) \leq f(x_0)- f(x_\lambda),
        \end{equation}
        and
        \begin{equation}
            \label{eq:eke-2}
            f(x)+\lambda d(x_\lambda, x) > f(x_\lambda), \quad \forall x \neq x_\lambda.
        \end{equation}
    \end{theorem}
    \begin{proof}
        Fix $x_0 \in \dom f$ and $\lambda>0$. Define
        $$
            S(x) := \{y:\  f(y) \le f(x) -\lambda d(y,x)\}.
        $$
        Because $f$ is lower semicontinuous, $S(x)$ is closed for all $x\in X$. It is easy to check that $S$ is idempotent. Indeed, if $y\in S(x)$ and $z\in S(y)$ this means by definition that
        $$
            f(z) \le f(y) - \lambda d(z,y) \le f(x) - \lambda d(y,x) - \lambda d(z,y) \le f(x) - \lambda d(z,x),
        $$
        the latter by triangle inequality.
        Next, let $\{x_i\}_{i=0}^n$ be any finite $S$-orbit starting at $x_0$ with $n\ge1$. By definition, $d(x_{i+1},x_i) \le \lambda^{-1}(f(x_i)-f(x_{i+1}))$, and summing these we get
        \begin{equation}
            \label{eq:orb-len-est}
            \sum_{i=0}^{n-1} d(x_{i+1},x_i) \le \lambda^{-1}(f(x_0)-f(x_n)).
        \end{equation}
        Since the latter is $\le \lambda^{-1}(f(x_0)-\inf f)$, we see that each orbit starting at $x_0$ has a length at most $\lambda^{-1}(f(x_0)-\inf f)$, that is, (\emph{b}) of Theorem~\ref{thm:loev-mon} is impossible for this $S$. So, there is an $S$-orbit starting at $x_0$ and ending at $x_\lambda $ such that $S(x_\lambda)\subset \{x_\lambda\}$. The latter implies \eqref{eq:eke-2}.

        Taking, if necessary, a limit as $n\to\infty$ in \eqref{eq:orb-len-est}, and using the lower semicontinuity of $f$ and the triangle inequality, we get \eqref{eq:eke-1}.
    \end{proof}

    We conclude this section with a demonstration how LOEV principle can be used for proving other two famous theorems.

    \begin{theorem}(Oettli and Th\'era~\cite{Oettli_Thera})
        \label{thm:thera}
        Let $(X,d)$ be a complete metric space. Let
        $p:X\times X\to \mathbb{R}\cup\{\infty\}$
        be a generalized pseudometric on $X$, that is $p(x,x) = 0$ for all $x\in X$, and 
        \begin{equation}
            \label{eq:sem-2}
            p(x,z) \le p(x,y) + p(y,z),\quad\forall x,y,z\in X,
        \end{equation}
        which is lower semicontinuous on its second argument and also for some $x_0\in X$
        \begin{equation}
            \label{eq:otth-1}
            \inf_{x\in X} p(x_0,x) > -\infty.
        \end{equation}
        Denote $
            A := \{ x\in X:\ p(x_0,x) + d(x_0,x) \le 0\}$. 
            Let $\Psi\subset X$ be such that for each $x\in A\setminus\Psi$ there exists $y\in X$ such that $y\neq x$ and
        \begin{equation}
            \label{eq:otth-Psi}
            p(x,y) + d(x,y) \le 0.
        \end{equation}
        Then $A\cap\Psi\neq\emptyset$. 
        \end{theorem}
    \begin{proof}
        Here  condition \eqref{eq:otth-Psi} directly suggests the definition of $S$:
        $$
            S(x) := \{y:\ p(x,y) + d(x,y) \le 0\}.
        $$
        Because $p+d$ satisfies the triangle inequality, $S$ is idempotent, and because the function $y\to p(x,y) + d(x,y)$ is lower semicontinuous, $S(x)$ is closed for all $x$. We apply Theorem~\ref{thm:loev-mon} to $S$ at $x_0$.

        If $x_0,x_1,\ldots, x_n$, where $n\ge1$, is an $S$-orbit, we have $d(x_i,x_{i+1}) \le -p(x_i,x_{i+1})$ for all $i=0,\dots,n-1$, so, using \eqref{eq:sem-2} and \eqref{eq:otth-1} we get
        $$
            \sum_{i=0}^{n-1}d(x_i,x_{i+1}) \le -\sum_{i=0}^{n-1}p(x_i,x_{i+1}) \le -p(x_0,x_n) \le -\inf_{x\in X}p(x_0,x) < \infty,
        $$
        so, (\emph{b}) is impossible. Therefore, there is some $S$-orbit of finite length $\{x_n\}_{n\ge0}$ ending at some $\bar x\in X$ such that $S(\bar x) \subset\{\bar x\}$.

        Since $x_0\in A$, we get by induction that $\{x_n\}_{n\ge0} \subset A$. Since $A$ is closed, we have that $\bar x\in A$. But then \eqref{eq:otth-Psi} shows that $\bar x\in\Psi$.
    \end{proof}
    \begin{theorem}(Fabi\'an and Preiss~\cite{Fabian_Preiss})
        \label{thm:fap}
        Let $(X,d)$ be a complete metric space. Let $I$ be an index set. Let $\{p_i\}_{i\in I}$ be pseudometrics on $X$ continuous as functions from  $X\times X$ to $\R$, and let $\{f_i\}_{i\in I}$ be lower semicontinuous functions from $X$ to $\mathbb{R}\cup\{+\infty\}$ with the following properties:
        \begin{equation}
            \label{eq:fap-i0}
            \exists i_0\in I:\quad p_{i_0} = d.
        \end{equation}
        For some fixed 
        $$
            x_0\in \bigcap_{i\in I} \dom f_i
        $$
        and for the set
        $$
            \Phi := \{x:\ p_i(x_0,x) \le f_i(x_0) - f_i(x),\ \forall i\in I\},
        $$
        it is fulfilled that
        \begin{equation}
            \label{eq:fap-next-y}
            \forall x\in \Phi:\ f_{i_0}(x) > 0,\ \exists y\neq x:\ p_i(x,y) \le f_i(x) - f_i(y),\ \forall i\in I.
        \end{equation}
        Then there exists $
              \bar x \in \Phi$ such that $f_{i_0}(\bar x) \le 0$.
    \end{theorem}
    \begin{proof}
        The lower semicontinuity of $f_i$ ensures that the set $\Phi$ is closed. Also, the $y$ in \eqref{eq:fap-next-y} is actually in $\Phi$. Indeed, $p_i(x_0,y) \le p_i(x_0,x) + p_i(x,y) \le (f_i(x_0)-f_i(x)) + (f_i(x)-f_i(y)) = f_i(x_0) - f_i(y)$ for all $i\in I$. Thus we can  restrict our considerations to the complete metric space $(\Phi,d)$.

        Define for $x\in \Phi$
        $$
            S(x) := \{y\in\Phi:\ p_i(x,y) \le f_i(x) - f_i(y),\ \forall i\in I\}.
        $$
        It is easy to check as we did above, that $S$ is idempotent and has closed values. Clearly, for each finite $S$-orbit $\{x_j\}_{j=0}^n$ with $n\ge1$, using that $p_{i_0}=d$, we have that
        $$
            \sum_{j=0}^{n-1} d(x_{j},x_{j+1}) \le f_{i_0}(x_0) - f_{i_0}(x_n),
        $$
        so if there is an infinite length $S$-orbit starting at $x_0$ then $f_{i_0}$ will become negative at some point. If, on the other hand, some finite length $S$-orbit ends at $\bar x$ such that $S(\bar x)\subset \{\bar x\}$, then \eqref{eq:fap-next-y} shows that $f_{i_0}(\bar x) \le 0$.
    \end{proof}

\section{Application: perturbed minimization on a $G_\delta$ set}\label{sec:gd}

    This section is inspired by certain classical results of \v{C}oban, Kenderov and Revalski, see e.g.~\cite{Coban_Kenderov_Revalski}. Those concern the question of when it is possible to achieve minimization by perturbing with a continuous function. Actually the focus is on when the latter can be achieved by a generic set of continuous functions, but we leave this aspect aside and rather construct something more explicit in a special case.

    Let $(X,d)$ be a complete metric space and let $Y\subset X$ be nonempty and $G_\delta$:
    $$
        Y = \bigcap_{i=1}^\infty U_n,\text{ where }U_n\text{'s are open.}
    $$
    Since Ekeland Theorem is equivalent to completeness, unless $Y$ is closed, it is pointless to search for a Lipschitz minimizing perturbation in general, but finding -- for a given proper and bounded below lower semicontinuous function $f:Y\to\mathbb{R}\cup\{+\infty\}$ -- a continuous function $g$ with a small uniform norm and such that $f+g$ attains minimum is easy. Indeed, one can just pick an $\varepsilon$-minimum to $f$ and use the modulus of semicontinuity of $f$ at that point to construct $g$. However, in this construction the continuity of $g$ at the minimum point will depend on the function $f$. Instead, we will provide below a premetric $g$ on $Y$, such that $(Y,\tau_d,g)$ is a  $\Sigma_g$-complete premetric space and use Theorem~\ref{th:Ekeland}. In this way the local continuity of the perturbing function depends only on the local geometry of the set $Y$, but not on the specific function to be perturbed. Here are the details.

    First, by considering instead of $d$ the equivalent metric $(x,y)\to d(x,y)/(1+d(x,y))$ we may and do assume that the metric is bounded. We do this in order to have a bounded perturbation. Next, set
     $$
        F_n := X\setminus U_n.
    $$

    Define for $n\in\mathbb{N}$ and $x\in U_n$
    $$
        \varphi_n (x) := \frac{1}{d(x,F_n)}.
    $$
    Obviously, $\varphi_n$ is continuous on $U_n\supset Y$. 
    Define for $x,y\in Y$
    \begin{equation}
        \label{eq:gd-g-def}
        g(x,y) := d(x,y) + \sum_{n=1}^\infty \frac{|\varphi_n(x)-\varphi_n(y)|}{2^n(1+|\varphi_n(x)-\varphi_n(y)|)}.
    \end{equation}
    As an uniform limit of continuous functions, $g$ is continuous. So it satisfies (P2) from Section~\ref{sec:semicomplete}, while (P1) is obvious. Thus, $g$ is a premetric on $Y$. Obviously, $g$ is also bounded.

To prove that $(Y,\tau_d,g)$ is a $\Sigma_g$-complete premetric space we need two preliminary lemmas.

    \begin{lemma}
        \label{lem:gd}
        Let $a_i\ge 0$. Then
        $$
            \sum_{i=1}^\infty \frac{a_i}{1+a_i} < \infty \iff
            \sum_{i=1}^\infty a_i < \infty.
        $$
    \end{lemma}
    \begin{proof}
        Since $a_i/(1+a_i) \le a_i$, the $(\Leftarrow)$ part is clear.

        If $\sum a_i/(1+a_i) < \infty$, then eventually $a_i/(1+a_i) < 1/2$, which means that $a_i < 1$, and, therefore, $a_i/(1+a_i) > a_i/2$ for all but finitely many $i$'s, so $\sum a_i < \infty$.
    \end{proof}
    \begin{lemma}
        \label{lem:gd-2}
        For each $n\in\mathbb{N}$ and $K>0$ the level set
        $$
            L_{n,K} := \{x\in U_n:\ \varphi_n(x) \le K\}
        $$
        is a closed subset of $X$.
    \end{lemma}
    \begin{proof}
        Fix $n$ and $K$ and let $\{x_i\}_{i=1}^\infty\subset L_{n,K}$ be such that $x_i\to x$ as $i\to\infty$. If we assume that $x\in F_n$, then $d(x_i,F_n) \le d(x_i,x)\to0$, as $i\to\infty$, so $\varphi_n(x_i)=1/d(x_i,F_n)\to\infty$, as $i\to\infty$, which contradicts $x_i\in L_{n,K}$, that is, $\varphi_n(x_i)\le K$, for all $i\in\mathbb{N}$.

        So, $x\in U_n$, and, since $\varphi_n$ is continuous at $x$, we have that $\varphi_n(x)\le K$.
    \end{proof}

Now we can prove the promised statement.
  \begin{proposition}
        \label{pro:gd}
         $(Y,\tau_d,g)$ is a $\Sigma_g$-complete premetric space.
    \end{proposition}

    \begin{proof}
        Let $\{x_i\}_{i\ge0}$ be such sequence that
        $$
            \sum_{i=0}^\infty g(x_{i+1},x_i) < \infty.
        $$
        By \eqref{eq:gd-g-def} we get
        \begin{eqnarray*}
            \infty &>& \sum_{i=0}^\infty \left(d(x_{i+1},x_i) + \sum_{n=1}^\infty \frac{|\varphi_n(x_{i+1})-\varphi_n(x_i)|}{2^n(1+|\varphi_n(x_{i+1})-\varphi_n(x_i)|)}\right)\\
            &=& \sum_{i=0}^\infty d(x_{i+1},x_i) + \sum_{n=1}^\infty\left(2^{-n}\sum_{i=0}^\infty\frac{|\varphi_n(x_{i+1})-\varphi_n(x_i)|}{1+|\varphi_n(x_{i+1})-\varphi_n(x_i)|}\right).
        \end{eqnarray*}
        In particular, $\sum_{i=0}^\infty d(x_{i+1},x_i) < \infty$, so $\{x_i\}_{i\ge0}$ is a Cauchy sequence, thus $x_i\to x\in X$, as $i\to\infty$. We have to show that $x\in Y$.

        For each fixed $n\in\mathbb{N}$ we have from the above inequality
        $$
            \sum_{i=0}^\infty \frac{|\varphi_n(x_i)-\varphi_n(x_{i+1})|}{1+|\varphi_n(x_i)-\varphi_n(x_{i+1})|} < \infty,
        $$
        so, Lemma~\ref{lem:gd} gives
        $$
            \sum_{i=0}^\infty |\varphi_n(x_i)-\varphi_n(x_{i+1})| < \infty,
        $$
        which means that the sequence $\{\varphi_n(x_i)\}_{i\ge0}$ is bounded by, say, $K$. In other words, the sequence $\{x_i\}_{i\ge0}$ is fully contained in the closed level set $L_{n,K}$, see Lemma~\ref{lem:gd-2}. But then its limit point $x$ is also in $L_{n,K}\subset U_n$, implying that $x\in U_n$. Since $n$ was arbitrary, $x\in Y$.
    \end{proof}
    By applying Theorem~\ref{th:Ekeland}, we immediately get our perturbation.
    \begin{corollary}
        \label{cor:gd}
   Let $(X,d)$ be a complete metric space and $Y\subset X$ be a nonempty $G_\delta$ set. Let $g$ be given by \eqref{eq:gd-g-def}.     Let  $f:Y\to\mathbb{R}\cup\{+\infty\}$ be proper and bounded below lower semicontinuous function. Then for each $\varepsilon > 0$ there is $x_\varepsilon\in Y$ such that
        $$
            f(x) + \varepsilon g(x_\varepsilon, x) \ge f(x_\varepsilon),\quad\forall x\in Y.
        $$
    \end{corollary}

\section*{Acknowledgements}
The authors dedicate this article to Prof. Michel Th\'era in acknowledgement of his pioneering, and with gratitude for his constant encouragement and support.

\section*{Funding}
The study  is supported by  the European Union-NextGenerationEU, through the National Recovery and Resilience Plan of the Republic of Bulgaria,  project  SUMMIT BG-RRP-2.004-0008-C01.

\end{document}